\documentclass{amsart}

\usepackage{amsmath,amssymb,amsfonts}
\usepackage{ifthen}
\usepackage{hyperref}

\newtheorem{theorem}{Theorem}[section]
\newtheorem{corollary}[theorem]{Corollary}
\newtheorem{lemma}[theorem]{Lemma}
\theoremstyle{definition}

\newtheorem{remark}{Remark}

\newcommand{\eps}{\varepsilon}
\DeclareMathOperator*{\argmin}{arg\,min}
\DeclareMathOperator{\range}{Ran}
\newcommand{\field}[1]{\mathbb{#1}}
\newcommand{\R}{\field{R}}
\newcommand{\N}{\field{N}}
\DeclareMathOperator{\hyp}{hyp}
\DeclareMathOperator{\len}{length}
\DeclareMathOperator{\area}{area}
\DeclareMathOperator{\inn}{int}
\DeclareMathOperator{\Emb}{Emb}

\newcommand{\inner}[3][n]{\SwitchBracketsizeLeft{#1}\LeftBracketSize\langle#2,#3\SwitchBracketsizeRight{#1}\RightBracketSize\rangle}
\newcommand{\abs}[2][n]{\SwitchBracketsizeLeft{#1}\LeftBracketSize\lvert#2\SwitchBracketsizeRight{#1}\RightBracketSize\rvert}
\newcommand{\norm}[2][n]{\SwitchBracketsizeLeft{#1}\LeftBracketSize\lVert#2\SwitchBracketsizeRight{#1}\RightBracketSize\rVert}
\newcommand{\set}[3][b]{\SwitchBracketsizeLeft{#1}\LeftBracketSize\{#2:#3\SwitchBracketsizeRight{#1}\RightBracketSize\}}

\newcommand{\NextScriptStyle}[1]{{\scriptstyle{#1}}}
\newcommand{\NextScriptScriptStyle}[1]{{\scriptscriptstyle{#1}}}
\newcommand{\NextTextStyle}[1]{{\textstyle{#1}}}
\newcommand{\NextDisplayStyle}[1]{{\displaystyle{#1}}}
\newcommand{\SwitchBracketsizeLeft}[1]{
  \ifthenelse{\equal{#1}{b}\OR\equal{#1}{big}}{\let\LeftBracketSize=\bigl}{
    \ifthenelse{\equal{#1}{B}\OR\equal{#1}{Big}}{\let\LeftBracketSize=\Bigl}{
      \ifthenelse{\equal{#1}{g}\OR\equal{#1}{bigg}}{\let\LeftBracketSize=\biggl}{
    \ifthenelse{\equal{#1}{G}\OR\equal{#1}{Bigg}}{\let\LeftBracketSize=\Biggl}{
      \ifthenelse{\equal{#1}{s}\OR\equal{#1}{small}}{\let\LeftBracketSize=\NextScriptStyle}{
        \ifthenelse{\equal{#1}{ss}}{\let\LeftBracketSize=\NextScriptScriptStyle}{
          \ifthenelse{\equal{#1}{t}\OR\equal{#1}{text}}{\let\LeftBracketSize=\NextTextStyle}{
        \ifthenelse{\equal{#1}{d}\OR\equal{#1}{display}}{\let\LeftBracketSize=\NextDisplayStyle}{
          \ifthenelse{\equal{#1}{a}\OR\equal{#1}{auto}}{\let\LeftBracketSize=\left}{
            \let\LeftBracketSize=\relax}}}}}}}}}}
\newcommand{\SwitchBracketsizeRight}[1]{
  \ifthenelse{\equal{#1}{b}\OR\equal{#1}{big}}{\let\RightBracketSize=\bigr}{
    \ifthenelse{\equal{#1}{B}\OR\equal{#1}{Big}}{\let\RightBracketSize=\Bigr}{
      \ifthenelse{\equal{#1}{g}\OR\equal{#1}{bigg}}{\let\RightBracketSize=\biggr}{
    \ifthenelse{\equal{#1}{G}\OR\equal{#1}{Bigg}}{\let\RightBracketSize=\Biggr}{
      \ifthenelse{\equal{#1}{s}\OR\equal{#1}{small}}{\let\RightBracketSize=\NextScriptStyle}{
        \ifthenelse{\equal{#1}{ss}}{\let\RightBracketSize=\NextScriptScriptStyle}{
          \ifthenelse{\equal{#1}{t}\OR\equal{#1}{text}}{\let\RightBracketSize=\NextTextStyle}{
        \ifthenelse{\equal{#1}{d}\OR\equal{#1}{display}}{\let\RightBracketSize=\NextDisplayStyle}{
          \ifthenelse{\equal{#1}{a}\OR\equal{#1}{auto}}{\let\RightBracketSize=\right}{
            \let\RightBracketSize=\relax}}}}}}}}}}

\title[Convergence Rates for a Non-differentiable Operator]{An Application of Source Inequalities 
for Convergence Rates of Tikhonov Regularization 
with a Non-differentiable Operator}

\author{Markus Grasmair}
\address{Computational Science Center, University of Vienna, Vienna, Austria}

\begin{document}

\begin{abstract}
In this paper we study Tikhonov regularization
for the stable solution of an ill-posed non-linear
operator equation.
The operator we consider, which is related
to an active contour model for image segmentation,
is continuous, compact, but nowhere differentiable.
Nevertheless we are able to derive convergence
rates under different smoothness assumptions on
the true solution by employing the method
of variational or source inequalities.
With this approach, we can prove up to linear
convergence with respect to the norm.
\bigskip

\bigskip

\end{abstract}

\maketitle

\section{Introduction}

The study of convergence rates within the context of
solutions of ill-posed problems is concerned with the question
of how good one can approximate the true solution of an equation
given some noisy data of known noise level $\delta$.
Assume that $A\colon X \to Y$ is a bounded linear operator between two Hilbert
spaces $X$ and $Y$ and we are trying to solve the linear operator
equation $Ax = y^\delta$ with noisy right hand side $y^\delta$
satisfying $\norm{y-y^\delta} \le \delta$ for some ``true'' data $y \in Y$.
Then the classical Tikhonov approach consists
in the minimization of the functional
\begin{equation}\label{eq:A}
\norm{Ax-y^\delta}_Y^2 + \alpha\norm{x}_X^2
\end{equation}
with some regularization parameter $\alpha > 0$.
If $\delta$ and $\alpha$ tend to zero in such a way that $\delta^2/\alpha \to 0$,
then the minimizers $x_\alpha^\delta$ of~\eqref{eq:A} will
converge to the ``true'' solution
\[
x^\dagger = A^\dagger y
\]
with $A^\dagger$ being the Moore--Penrose inverse of $A$ \cite{Gro84}.
Moreover, it is possible to estimate the quality of the approximation,
that is, the difference $\norm{x_\alpha^\delta-x^\dagger}$,
provided one has additional a-priori knowledge about smoothness properties
of $x^\dagger$:
If $x^\dagger = (A^*A)^\nu\omega$ for some $\omega \in X$ and $0 < \nu \le 1$,
then one can show for the parameter choice
$\alpha \sim \delta^{\frac{2}{2\nu+1}}$ that we have
\cite[Corollary~3.1.4]{Gro84}
\begin{equation}\label{eq:ratelin}
\norm{x_\alpha^\delta - x^\dagger} = O(\delta^{\frac{2\nu}{2\nu+1}})
\qquad\text{ as } \delta \to 0.
\end{equation}

In the case of a non-linear operator equation $F(x) = y$,
the classical approach to the derivation of convergence rates
uses a linearization of $F$ at the true solution $x^\dagger$.
Then, analogous convergence rates as in the linear case
are derived under the assumption of a source condition of the form
\begin{equation}\label{eq:sourcenonlin}
\bigl(F'(x^\dagger)^*F'(x^\dagger)\bigr)^\nu \omega = x^\dagger.
\end{equation}
In the proofs of these convergence rates, however,
it is necessary to introduce additional assumptions
coupling the non-linear functional $F$ and its linearization at $x^\dagger$.

An alternative approach to non-linear problems has been
suggested in~\cite{HofKalPoeSch07}.
Instead of linearizing the problem and using source equalities,
they proposed to use a single \emph{source inequality} of the form
\begin{equation}\label{eq:sourcei}
\norm{x-x^\dagger}_X^2
\le \norm{x}^2_X-\norm{x^\dagger}_X^2 + \beta_2\norm{F(x)-F(x^\dagger)}_Y.
\end{equation}
With this assumption, they were able to prove a convergence rate
in the norm of order $O(\sqrt{\delta})$.
This corresponds to the case $\nu = 1/2$ in~\eqref{eq:ratelin};
indeed, in the linear case, the source inequality~\eqref{eq:sourcei}
is equivalent to the source condition $x^\dagger = (A^*A)^{\frac{1}{2}}\omega$
(see~\cite{HofKalPoeSch07}).
Moreover, the approach has been generalized in~\cite{BotHof10,Gra10b}
to intermediate convergence rates, roughly corresponding to
the source condition $x^\dagger = (A^*A)^\nu \omega$ with $0 \le \nu \le 1/2$.

Apart from the possibility of treating non-linear operators,
source inequalities have the advantage that they do not require
a Hilbert space structure of the set $X$;
indeed, the results in~\cite{HofKalPoeSch07} were formulated in
a Banach space setting (\cite{Gra10b} does not even require linear spaces).
One problem in Banach spaces is that fractional powers of a linear operator
are not defined any more; a source condition of the form~\eqref{eq:sourcenonlin}
cannot even be formulated.
There is, however, a different approach to convergence rates of general
order for non-linear problems in Banach spaces by means of
\emph{approximate source conditions} (see~\cite{Hei08,Hei09,HeiHof09,Hof06}).
Still this approach requires a prior linearization of the operator.

While theoretically the approach by source inequalities can
be used for deriving convergence rates in general non-smooth situations,
most cases where it is been applied were differentiable settings
with a weak coupling between the operator and its linearization.
In this paper we will show by means of a concrete example that the
method of source inequalities can also be used for deriving
convergence rates under meaningful conditions for a non-differentiable
operator. We consider a situation that is closely related
to an active contour model for image segmentation, where the operator,
in its natural Hilbert space setting, is H\"older continuous but nowhere
differentiable.
Nevertheless it is possible to derive convergence rates up to order $O(\delta)$
under the assumption of sufficient smoothness of the true solution
(in the sense that it lies in a certain Sobolev space).

\section{Setting}

Denote by $H^1_+(S^1)$ the non-negative cone in the Sobolev space $H^1(S^1)$.
That is, $\gamma \in H^1_+(S^1)$ if and only if
$\gamma \in H^1(S^1)$ and $\gamma(t) \ge 0$ for every $t \in S^1$.
If $\gamma \in H_+^1(S^1)$, we denote by
\[
\hyp(\gamma) := \set{(t,x) \in S^1\times\R_{\ge 0}}{0 \le x \le \gamma(t)}
\]
the \emph{hypograph} of $\gamma$.
That is, $\hyp(\gamma)$ consists of all points lying
between the $t$-axis and the graph of $\gamma$.
We now define the mapping $F\colon H^1_+(S^1) \to L^2(S^1\times \R_{\ge 0})$,
\[
F(u) := \chi_{\hyp(\gamma)}
\]
with
\[
\chi_{\hyp(\gamma)}(t,x) := 
\begin{cases}
  0 &\text{if } \gamma(t) > x,\\
  1 &\text{if } 0 \le \gamma(t) \le x,
\end{cases}
\]
which maps the curve $\gamma$ to the characteristic function
of its hypograph.

Note that every function $\gamma\in H_+^1(S^1)$ is 
bounded and continuous, and therefore 
$\hyp(\gamma) \subset S^1\times \R_{\ge 0}$ is a (non-empty) compact set.
In particular this implies that the mapping $F$ indeed takes values in the
space $L^2(S^1\times\R_{\ge 0})$.
Moreover it is obvious that $F$ is injective.

We now consider the problem of solving the equation
$F(\gamma) = u$ for some given $u \in \range(F)$ in a stable way.
In addition, we assume that, instead of $u$, we are only given noisy data
$u^\delta\in L^2(S^1\times\R_{\ge 0})$ satisfying
$\norm{u^\delta-u}_{L^2} \le \delta$, where $\delta > 0$ is some
known noise level.
Here we do not assume that the noisy data lie in the range of $F$
or even that they are the characteristic function of some (Borel) set.

In order to find an approximate solution of the equation
\begin{equation}\label{eq:eq}
F(\gamma) = u^\delta,
\end{equation}
we apply classical (quadratic) Tikhonov regularization
using the squared homogeneous first order Sobolev norm as regularization term.
That is, we consider the Tikhonov functional
\[
\mathcal{T}_\alpha(\gamma;u^\delta)
:= \norm{F(\gamma)-u^\delta}_{L^2}^2 + \alpha\norm{\dot{\gamma}}_{L^2}^2
\]
and define, for $\alpha > 0$,
an approximate solution of~\eqref{eq:eq} by
\[
\gamma_\alpha^\delta :\in \argmin\set{\mathcal{T}_\alpha(\gamma;u^\delta)}{\gamma \in H^1_+(S^1)}.
\]

\subsection*{Relation to Active Contours}

One of the basic problems in image processing is segmentation:
Given a possibly noisy image, interpreted as a function
$u \in L^2(\R^2)$, the task in its simplest form is to partition this image
into two regions, one representing objects in the image,
the other representing the background.
To that end, Chan and Vese \cite{ChaVes01} proposed a variational method
based on the assumption that the objects (which are assumed to form a simply connected region)
share a common grey value $c_1$,
and also the background has a uniform grey value $c_2 \neq c_1$.
Without loss of generality we may assume that $c_2 = 1$ and $c_1 = 0$.
Then the method proposed in~\cite{ChaVes01} consists in minimizing,
for some regularization parameters $\alpha_1$, $\alpha_2 > 0$,
the functional
\begin{equation}\label{eq:ChanVese}
  I(\gamma) := \alpha_1 \len(\gamma)+\alpha_2\area(\gamma)
  + \int_{\inn(\gamma)}\hskip-10pt (1-u(x))^2\,dx 
  + \int_{\R^2\setminus\inn(\gamma)}\hskip-15pt u(x)^2\,dx.
\end{equation}
Here $\len(\gamma)$ denotes the length of the curve $\gamma$,
$\inn(\gamma) \subset \R^2$ denotes the set of all points lying inside
the curve $\gamma$, and $\area(\gamma)$ the size of the area inside the curve.
This functional is minimized over the space $\Emb(S^1;\R^2)$ of all closed curves
continuously embedded in $\R^2$.

Defining the function $F\colon \Emb(S^1;\R^2) \to L^2(\R^2)$,
\[
F(\gamma) := \chi_{\inn(\gamma)},
\]
one can write the functional equivalently as
\[
I(\gamma) = \alpha_1 \len(\gamma)+\alpha_2\area(\gamma) + \norm{F(\gamma)-u}_{L^2}^2.
\]
Moreover, for $\gamma$ sufficiently smooth, one can write
\[
\len(\gamma) = \int_{S^1} \abs{\dot{\gamma}(t)}\,dt,
\qquad\qquad
\area(\gamma) = \frac{1}{2}\int_{S^1} \inner{\gamma(t)}{\dot{\gamma}(t)}\,dt.
\]

Thus the Chan--Vese model~\eqref{eq:ChanVese} can
be brought in a quite similar form as the problem we consider in this paper,
which, conversely, can be interpreted as a simplified active contour model
for images on a half-cylinder.
A major reason for our simplification is the fact that it allows
us to work in a Hilbert space setting.
In contrast, the ``natural'' setting for the Chan--Vese model would
rather be the Banach manifold of Lipschitz embeddings of $S^1$ in $\R^2$
modulo reparameterizations of $S^1$
(note that $I(\gamma)$ is invariant with respect
to reparameterizations of $\gamma$; in order to obtain any
uniqueness result it is therefore necessary to factor out
reparameterizations).

\section{Well-posedness of the Regularization Method}

We now prove that quadratic Tikhonov regularization
yields a well-posed regularization method for the functional $F$.
To that end we mainly have to investigate the continuity properties
of $F$.

\begin{lemma}\label{le:hoelder}
  The mapping $F$ is H\"older continuous of degree $1/2$
  as a mapping from $L^2_+(S^1)$ to $L^2(S^1\times\R_{\ge 0})$.
  In particular, it is weakly continuous and compact
  as a mapping from $H^1_+(S^1)$ to $L^2(S^1\times\R_{\ge 0})$.
\end{lemma}

\begin{proof}
  Let $\gamma_1$, $\gamma_2 \in L^2_+(S^1)$.
  Then
  \begin{multline*}
    \norm{F(\gamma_1)-F(\gamma_2)}_{L^2}^2
    = \int_{S^1}\int_{\R\ge 0} \abs{\chi_{\hyp(\gamma_1)}(t,x)-\chi_{\hyp(\gamma_2)}(t,x)}^2\,dx\,dt\\
    = \int_{S^1} \abs{\gamma_1(t)-\gamma_2(t)}\,dt
    = \norm{\gamma_1-\gamma_2}_{L^1}
    \le \sqrt{2\pi}\norm{\gamma_1-\gamma_2}_{L^2},
  \end{multline*}
  which proves the H\"older continuity of $F$
  as a mapping from $L^2_+(S^1)$.
  The weak continuity and compactness of $F$ as a mapping
  from $H^1_+(S^1)$ now follow from the fact that
  $H^1(S^1)$ is compactly embedded in $L^2(S^1)$.
\end{proof}

\begin{theorem}\label{th:tikreg}
  The following hold:
  \begin{enumerate}
  \item For every $u \in L^2(S^1\times\R_{\ge 0})$ and $\alpha > 0$
    the functional $\mathcal{T}_\alpha(\cdot;u)$ attains its
    minimum in $H^1_+(S^1)$.
  \item Assume that $u \in L^2(S^1\times\R_{\ge 0})$ and $\alpha > 0$.
    Let $\{u^{(k)}\}_{k\in\N} \subset L^2(S^1\times\R_{\ge 0})$ be any
    sequence converging to $u$, and let
    \[
    \gamma^{(k)} \in \argmin\set{\mathcal{T}_\alpha(\gamma;u^{(k)})}{\gamma\in H^1_+(S^1)}.
    \]
    Then there exists a sub-sequence $\{\gamma^{(k_j)}\}_{j\in\N}$
    converging strongly in $H^1(S^1)$ to some
    \[
    \gamma \in \argmin\set{\mathcal{T}_\alpha(\gamma;u)}{\gamma\in H^1_+(S^1)}.
    \]
  \item Assume that $u = F(\gamma^\dagger)$ for some $\gamma^\dagger \in H^1_+(S^1)$.
    Let $\{\delta_k\}_{k\in\N} \subset \R_{>0}$ and $\{\alpha_k\}_{k\in\N} \subset \R_{>0}$ be 
    sequences satisfying $\delta_k \to 0$, $\alpha_k \to 0$, and $\delta_k^2/\alpha_k \to 0$,
    and let $u^{\delta_k} \in L^2(S^1\times\R_{\ge 0})$, $k\in\N$, satisfy $\norm{u_k-u}_{L^2} \le \delta_k$.
    Let
    \[
    \gamma^{(k)} := \gamma_{\alpha_k}^{\delta_k} 
    \in \argmin\set{\mathcal{T}_{\alpha_k}(\gamma;u^{\delta_k})}{\gamma\in H^1_+(S^1)}.
    \]
    Then the sequence $\{\gamma^{(k)}\}_{k\in\N}$ converges
    strongly in $H^1(S^1)$ to $\gamma^\dagger$.
  \end{enumerate}
\end{theorem}

\begin{proof}
  This follows from standard results in the theory of (non-linear)
  Tikhonov regularization, which can be found, for instance,
  in~\cite{EngHanNeu96,SchGraGroHalLen09}.
  The only (ever so slight) complication is the fact that the regularization 
  term uses the homogeneous Sobolev norm, which is
  not coercive on $H^1(S^1)$.
  The coercivity of the functional $\mathcal{T}_\alpha$, however,
  follows immediately from the chain of inequalities
  \[
  \norm{F(\gamma)-u}_{L^2}^2 \le 2\norm{F(\gamma)}_{L^2}^2+2\norm{u}_{L^2}^2
  = 2\norm{\gamma}_{L^1} + 2\norm{u}_{L^2}^2
  \le 2\norm{\gamma}_{L^2}+2\norm{u}_{L^2}^2.
  \]
\end{proof}

\begin{remark}
  Because $F$ is injective, it follows that the equation
  $F(\gamma) = u$ can have at most one solution.
  Still it is possible that the optimisation problem
  $\mathcal{T}_\alpha(\gamma;u^\delta) \to \min$ can have multiple solutions,
  even for $\delta > 0$ arbitrarily small.

  Consider for instance the situation where
  $\gamma^\dagger(t) = 1$ and $F(\gamma^\dagger) = \chi_{S^1\times [0,1]}$.
  Define for $\delta > 0$ with $ \delta^2 < \pi$ the function
  \[
  u^\delta(t,x) :=
  \begin{cases}
    1 & \text{if } 0 \le x < 1-\delta^2/\pi,\\
    \dfrac{1}{2} & \text{if } \abs{x-1} \le \delta^2/\pi,\\
    0 & \text{if } x > 1+\delta^2/\pi.
  \end{cases}
  \]
  Then $\norm{u^\delta-u}_{L^2} = \delta$.

  Now let $\gamma \in H^1_+(S^1)$ and define
  \[
  J^+(\delta) := \set{t \in S^1}{\gamma(t) > 1+\delta^2/\pi},
  \quad\quad
  J^-(\delta) := \set{t \in S^1}{\gamma(t) < 1-\delta^2/\pi}.
  \]
  Then it is easy to see that
  \[
  \norm{F(\gamma)-u^\delta}_{L^2}^2
  = \delta^2 + \int_{J^+(\delta)}\hskip-10pt \gamma(t)-1-\delta^2/\pi\,dt
  + \int_{J^-(\delta)}\hskip-10pt 1-\gamma(t)-\delta^2/\pi\,dt.
  \]
  Therefore $\norm{F(\gamma)-u^\delta}_{L^2}^2$ is minimal
  (with value $\delta^2$) if and only if 
  $\abs{\gamma(t)-1} \le \delta^2/\pi$ for every $t$.
  Thus it follows that
  \[
  \mathcal{T}_\alpha(\gamma;u^\delta) \ge \delta^2
  \]
  for every $\gamma \in H^1_+(S^1)$, and equality holds if
  and only if $\gamma$ is a constant function of the form
  $\gamma(t) = 1+c$ for some $\abs{c} \le \delta^2/\pi$.
  In other words, \emph{every} such function is a minimizer
  of the Tikhonov functional with data $u^\delta$ and
  \emph{any} regularization parameter $\alpha > 0$.

  In particular, this example shows that the stability
  result in Theorem~\ref{th:tikreg} (item (2)) really
  requires the formulation in terms of sub-sequences;
  it can happen that the sequence $\{\gamma^{(k)}\}_{k\in\N}$
  itself does not converge.
  Note, however, that it is also possible to formulate
  the stability result in terms of set convergence.
  Such an approach has for instance been used
  in~\cite{GraHalSch11b}.
\end{remark}

\section{Convergence Rates}

In this main section of the paper we will discuss the derivation
of convergence rates, that is, quantitative estimates for the
difference between the regularized solution $\gamma_\alpha^\delta$
and the true solution $\gamma^\dagger$ in dependence of $\alpha$ and $\delta$.

Throughout this whole section we assume that the equation
$F(\gamma) = u$ has a (necessarily unique) solution $\gamma^\dagger\in H^1_+(S^1)$.
Moreover we denote by $\gamma_\alpha^\delta$ \emph{any} minimizer
of the regularization functional $\mathcal{T}_\alpha(\cdot;u^\delta)$
for \emph{any} noisy data $u^\delta \in L^2(S^1\times\R_{\ge 0})$ satisfying
$\norm{u^\delta-u}_{L^2} \le \delta$.
\medskip

Classically, convergence rates for quadratic Tikhonov regularization
of non-linear operators on Hilbert spaces have been derived
for sufficiently smooth operators under the assumption that
the true solution $\gamma^\dagger$ is contained in the range
of the adjoint of the derivative of the operator $F$ at $\gamma^\dagger$.
In addition it is necessary to compensate for the non-linearity
of the operator $F$ by assuming additional regularity properties
of its derivative $F'$.
One example is the following theorem taken from~\cite{SchGraGroHalLen09}.
Similar results can also be found in~\cite{EngHanNeu96,EngKunNeu89}.

\begin{theorem}
  Assume that $F$ is G\^ateaux differentiable in a neighborhood
  of $\gamma^\dagger$ and that there exist
  $\omega \in L^2(S^1\times\R_{\ge 0})$ and $c > 0$
  with $c\norm{\omega}_{L^2} < 1$
  such that
  \[
  \gamma^\dagger = F'(\gamma^\dagger)^*\omega
  \]
  and
  \[
  \norm[b]{F(\gamma)-F(\gamma^\dagger)-F'(\gamma^\dagger)(\gamma-\gamma^\dagger)}_{L^2}
  \le \frac{c}{2}\norm{\gamma-\gamma^\dagger}^2_{H^1},
  \]
  for every $\gamma\in H^1_+(S^1)$ sufficiently close to $\gamma^\dagger$.
  Then we have for a parameter choice $\alpha\sim \delta$ the rate
  \[
  \norm{\gamma_\alpha^\delta-\gamma^\dagger}_{H^1} = O(\sqrt{\delta}).
  \]
\end{theorem}

The crucial assumption in this theorem is the
\emph{source condition} $\gamma^\dagger \in F'(\gamma^\dagger)^*\omega$,
which links the non-linear problem to the much better understood
linear theory. The additional assumptions mainly guarantee
that the gap between the non-linear and the linearized problem
is sufficiently small.
Still, the whole approach for the derivation of convergence
rates relies on a linearization of the operator $F$,
which in our situation is not easily possible,
as the next result shows.

\begin{lemma}\label{le:nondiff}
  The mapping $F\colon H^1_+(S^1)\to L^2(S^1\times\R_{\ge 0})$
  is nowhere differentiable in the interior of its domain.
\end{lemma}

\begin{proof}
  First note that the interior of $H^1_+(S^1)$ consists
  of all functions $\gamma \in H^1(S^1)$ with $\gamma(t) > 0$
  for every $t \in S^1$.

  Let therefore $\gamma \in H^1_+(S^1)$ satisfy $\gamma(t) > 0$ everywhere,
  and let $\sigma \in H^1(S^1)$. Then, as shown in the proof of Lemma~\ref{le:hoelder},
  we have
  \[
  \norm{F(\gamma+s\sigma)-F(\gamma)}_{L^2}^2 = \int_{S^1}\abs{s\sigma(t)}\,dt.
  \]
  Now assume that $\sigma \neq 0$.
  Then there exist a non-empty open subset $J \subset S^1$ and $\eps > 0$
  such that $\abs{\sigma(t)} > \eps$ for every $t \in J$.
  Consequently
  \[
  \norm[b]{F(\gamma+s\sigma)-F(\gamma)}_{L^2}^2
  = \abs{s}\int_{S^1} \abs{\sigma(t)}\,dt \ge \eps \abs{s}\abs{J}.
  \]
  Thus
  \[
  \lim_{s \to 0^+} \frac{1}{s}\norm[b]{F(\gamma+s\sigma)-F(\gamma)}_{L^2}
  \ge \lim_{s\to 0^+} \sqrt{\frac{\eps\abs{J}}{s}} = +\infty,
  \]
  which proves that $F$ does not even have a one-sided directional
  derivative in any non-trivial direction.
\end{proof}

For the next result---the main theorem of this paper---recall
the definition of the fractional order Sobolev spaces
$W^{s,q}$, $s \in \R$, $1 \le q \le \infty$,
and the homogeneous Sobolev spaces $W_0^{s,q}$,
which can for instance be found in~\cite{BerLoe76}.

\begin{theorem}\label{th:main1}
  Assume that $\gamma^\dagger \in W^{s,q}(S^1)$ with $1 < s \le 2$
  and $2 \le q < +\infty$.
  Then there exist constants $c_1$, $c_2$, $c_3 > 0$ such that
  \begin{multline*}
  c_1 \norm{\gamma-\gamma^\dagger}_{H^1_0}^2\\
  \le \norm{\gamma}_{H^1_0}^2 - \norm{\gamma^\dagger}_{H^1_0}^2
  + c_2 \norm{F(\gamma)-F(\gamma^\dagger)}_{L^2}^{4-\frac{4}{s}}
  + c_3 \norm{F(\gamma)-F(\gamma^\dagger)}_{L^2}^{2-\frac{2q}{qs-s+1}}
  \end{multline*}
  for every $\gamma \in H^1_+(S^1)$.

  If $\gamma^\dagger \in W^{s,\infty}(S^1)$ with $1 < s \le 2$, then
  \begin{equation}\label{eq:ineqinfty}
  c_1 \norm{\gamma-\gamma^\dagger}_{H^1_0}^2 
  \le \norm{\gamma}_{H^1_0}^2 - \norm{\gamma^\dagger}_{H^1_0}^2
  + c_2 \norm{F(\gamma)-F(\gamma^\dagger)}_{L^2}^{4-\frac{4}{s}}.
  \end{equation}
\end{theorem}

\begin{proof}
  We only show the assertion for the more complicated case $q < \infty$.

  First note that
  \[
  \norm{\dot{\gamma}-\dot{\gamma}^\dagger}^2
  = \norm{\dot{\gamma}}^2 - \norm{\dot{\gamma}^\dagger}^2
  - 2\inner{\dot{\gamma}^\dagger}{\dot{\gamma}-\dot{\gamma}^\dagger}.
  \]
  In order to estimate the last term on the right hand side of
  this equation, we will use the abbreviation
  \[
  \sigma := \gamma-\gamma^\dagger.
  \]
  Because $\gamma^\dagger \in W^{s,q}_+(S^1)$, it follows that
  \[
  \abs[b]{\inner{\dot{\gamma}^\dagger}{\dot{\sigma}}}
  \le \norm{\gamma^\dagger}_{W^{s,q}_0}\norm{\sigma}_{W^{2-s,q_*}_0}
  \]

  Moreover the interpolation inequality for Sobolev functions implies that
  there exists some $C_1 > 0$ only depending on $s$ and $q_*$ such that
  \[
  \norm{\sigma}_{W^{2-s,q_*}_0} 
  \le C_1\norm{\sigma}_{L^{q_*}}^{s-1}\norm{\sigma}_{W^{1,q_*}_0}^{2-s}
  \le C_1\norm{\sigma}_{L^{q_*}}^{s-1}\norm{\sigma}_{H^1_0}^{2-s}.
  \]
  The last inequality holds, because $q \ge 2$ and therefore $q_* \le 2$.
  Now note that
  \[
  \begin{aligned}
  \norm{\sigma}_{L^{q_*}} & \le \norm{\sigma}_{L^1}^{\frac{1}{q_*}}\norm{\sigma}_{L^\infty}^{\frac{1}{q}}\\
  &\le \norm{\sigma}_{L^1}^{\frac{1}{q_*}} \bigl(\norm{\sigma}_{L^1}+\norm{\sigma}_{W^{1,1}_0}\bigr)^{\frac{1}{q}}\\
  &\le \norm{\sigma}_{L^1}+\norm{\sigma}_{L^1}^{\frac{1}{q_*}}\norm{\sigma}_{H^1_0}^{\frac{1}{q}}.
  \end{aligned}
  \]
  Setting $C_2 := 2C_1\norm{\gamma^\dagger}_{W^{s,q}_0}$,
  we thus obtain the estimate
  \[
  2\abs[b]{\inner{\dot{\gamma}^\dagger}{\dot{\sigma}}}
  \le C_2\norm{\sigma}_{L^1}^{s-1}\norm{\sigma}_{H^1_0}^{2-s}
  + C_2\norm{\sigma}_{L^1}^{\frac{s-1}{q_*}}\norm{\sigma}_{H^1_0}^{2-s+\frac{s-1}{q}}.
  \]
  Now we use Young's inequality $ab \le \frac{1}{p}a^p + \frac{1}{p_*}b^{p_*}$,
  which implies with $p=2/s$ and $p_* = 2/(2-s)$ that
  \begin{align*}
    C_2\norm{\sigma}_{L^1}^{s-1}\norm{\sigma}_{H^1_0}^{2-s}
    &\le \frac{s}{2}C_2^{\frac{2}{s}}\norm{\sigma}_{L^1}^{\frac{2s-2}{s}}+\Bigl(1-\frac{s}{2}\Bigr)\norm{\sigma}_{H^1_0}^2\\
    \intertext{and with $p = \frac{2q}{qs-s+1}$ and $p_* = \frac{2q}{2q-qs+s-1}$}
    C_2\norm{\sigma}_{L^1}^{\frac{s-1}{q_*}}\norm{\sigma}_{H^1_0}^{2-s+\frac{s-1}{q}}
    &\le \frac{qs-s+1}{2q}C_2^{\frac{2q}{qs-s+1}}\norm{\sigma}_{L^1}^{1-\frac{q}{qs-s+1}}\\
    &\qquad\qquad
    + \Bigl(1-\frac{qs-s+1}{2q}\Bigr)\norm{\sigma}_{H^1_0}^2.
  \end{align*}
  Thus we obtain with the constants
  \[
  \begin{aligned}
    c_1 &= \frac{(2q-1)(s-1)}{2q},\\
    c_2 &= \frac{s}{2}C_2^{\frac{2}{s}},\\
    c_3 &= \frac{qs-s+1}{2q}C_2^{\frac{2q}{qs-s+1}},\\
  \end{aligned}
  \]
  the estimate
  \[
  c_1\norm{\gamma-\gamma^\dagger}_{H^1_0}^2 \le \norm{\gamma}_{H^1_0}^2-\norm{\gamma^\dagger}_{H^1_0}^2
  + c_2 \norm{\gamma-\gamma^\dagger}_{L^1}^{2-\frac{2}{s}}
  + c_3 \norm{\gamma-\gamma^\dagger}_{L^1}^{1-\frac{q}{qs-s+1}}.
  \]
  Now the assertion follows from the fact that
  $\norm{F(\gamma)-F(\gamma^\dagger)}_{L^2}^2 = \norm{\gamma-\gamma^\dagger}_{L^1}$.
\end{proof}

\begin{corollary}
  Let the conditions of Theorem~\ref{th:main1} be satisfied
  and let
  \[
  \Phi(t) := c_2 t^{2\frac{s-1}{s}} + c_3 t^{\frac{s-1}{q_*}}
  \]
  in case $q < \infty$ and
  \[
  \Phi(t) := c_2 t^{2\frac{s-1}{s}}
  \]
  in case $q = \infty$.

  \begin{enumerate}
  \item If $q < \infty$ and $s < 2$ let $\Psi$ be the convex
    conjugate of the strictly convex mapping $t \mapsto \Phi^{-1}(2t)$.
    Then
    \[
    c_1 \norm{\gamma_\alpha^\delta-\gamma^\dagger}_{H^1_0}^2
    \le \frac{\delta^2}{\alpha}+\Phi(2\delta^2) + \frac{\Psi(\alpha)}{\alpha}.
    \]
  \item If $q = \infty$ and $s = 2$ we have
    \[
    c_1 \norm{\gamma_\alpha^\delta-\gamma^\dagger}_{H^1_0}^2 \le \frac{\delta^2}{\alpha} + 2c_2\delta^2
    \]
    whenever $\alpha \le 1/(2c_2)$.
  \end{enumerate}
\end{corollary}

\begin{proof}
  This is a direct consequence of Theorem~\ref{th:main1}
  and~\cite[Theorem~3.1]{Gra10b}.
\end{proof}

\begin{corollary}\label{co:rate}
  Let the conditions of Theorem~\ref{th:main1} be satisfied.
  \begin{enumerate}
  \item If $q < \infty$ or $s < 2$ we have for the parameter choice
    $\alpha \sim \delta^{2\frac{q_*-s+1}{q_*}}$ the convergence rate
    \[
    \norm{\gamma_\alpha^\delta - \gamma^\dagger}_{H^1_0} = O(\delta^{\frac{s-1}{q_*}}).
    \]
  \item If $q = \infty$ and $s = 2$ we have for a constant parameter
    choice $\alpha \le 1/(2c_2)$ the convergence rate
    \[
    \norm{\gamma_\alpha^\delta - \gamma^\dagger}_{H^1_0} = O(\delta).
    \]
  \end{enumerate}
\end{corollary}

\begin{proof}
  See~\cite[Corollary~3.1]{Gra10b}.
\end{proof}

\begin{remark}
  Note that the claim of Theorem~\ref{th:main1} in the
  case $s=2$ and $q = \infty$ can be shown much more easily
  with elementary computations.
  Indeed, in this case one has
  \[
  \begin{aligned}
  \norm{\gamma-\gamma^\dagger}_{H^1_0}^2
  & = \norm{\gamma}_{H^1_0}^2 - \norm{\gamma^\dagger}_{H^1_0}^2
  -2\int_{S^1} \dot{\gamma}^\dagger(t)\bigl(\dot{\gamma}(t)-\dot{\gamma}^\dagger(t)\bigr),dt\\
  & = \norm{\gamma}_{H^1_0}^2 - \norm{\gamma^\dagger}_{H^1_0}^2
  +2 \int_{S^1} \ddot{\gamma}^\dagger(t)\bigl(\gamma(t)-\gamma^\dagger(t)\bigr)\,dt\\
  & \le \norm{\gamma}_{H^1_0}^2 - \norm{\gamma^\dagger}_{H^1_0}^2
  +2\norm{\ddot{\gamma}^\dagger}_{L^\infty} \norm{\gamma-\gamma^\dagger}_{L^1}\\
  & = \norm{\gamma}_{H^1_0}^2 - \norm{\gamma^\dagger}_{H^1_0}^2
  + 2\norm{\ddot{\gamma}^\dagger}_{L^\infty} \norm{F(\gamma)-F(\gamma^\dagger)}_{L^2}^2.
  \end{aligned}
  \]
  Thus~\eqref{eq:ineqinfty} holds with $c_1 = 1$ and $c_2 = 2\norm{\ddot{\gamma}^\dagger}_{L^\infty}$.
\end{remark}

\section{Further Aspects}

\subsection{Differentiability}

In Lemma~\ref{le:nondiff} we have shown that $F$ is nowhere
differentiable when regarded as a mapping from
$H^1_+(S^1)$ to $L^2(S^1\times \R_{\ge 0})$.
We now discuss possible different settings, where $F$
has better regularity properties.

Because the range of $F$ consists only of characteristic functions
with compact support, it follows that $F$ can be regarded
as a mapping into any space $L^p(S^1\times \R_{\ge 0})$
with $1 \le p \le \infty$.
Moreover the same computation as in the proof of
Lemma~\ref{le:hoelder} shows that for every $1 \le p < \infty$
we have
\[
\norm{F(\gamma_1)-F(\gamma_2)}_{L^p}^p
= \norm{\gamma_1-\gamma_2}_{L^1} \le \sqrt{2\pi}\norm{\gamma-\gamma_2}_{L^2}.
\]
Thus, seen as a mapping $F\colon L^2(S^1) \to L^p(S^1;\R_{\ge 0})$,
the mapping $F$ is H\"older continuous of degree $1/p$;
in case $p = 1$, this shows that $F$ is Lipschitz continuous.
Note that for $p = \infty$, the mapping $F$ is discontinuous everywhere.
Concerning the differentiability of $F$, however, this change
of the target space does not really matter.
The same argumentation as in Lemma~\ref{le:nondiff} shows that
$F$ is nowhere differentiable as a mapping into $L^p(S^1\times\R_{\ge 0})$
with $1 < p < \infty$.

For $p = 1$ the situation is slightly different:
If $\gamma \in H^1_+(I)$ satisfies $\gamma(t) > 0$ for every $t$
and $\sigma \in H^1(I)$ is fixed, then 
the family of functions $\{\frac{1}{s}\bigl(F(\gamma+s\sigma)-F(\gamma)\bigr)\}_{s>0}$
forms a bounded subset of $L^1(S^1\times\R_{\ge 0})$.
Still, the limit $\lim_{s\to 0^+} \frac{1}{s}(F(\gamma+s\sigma)-F(\gamma))$
does not exist in $L^1(S^1\times\R_{\ge 0})$ and thus, again,
$F$ is nowhere directionally differentiable.

Finally, one can regard $F$ as a mapping from $H^1_+(S^1)$ to
the space $\mathcal{M}(S^1\times \R_{\ge 0})$ of finite Radon measures
on $S^1\times\R_{\ge 0}$.
Because $L^1(S^1\times\R_{\ge 0})$ is isometrically embedded into
$\mathcal{M}(S^1\times \R_{\ge 0})$, this does not change the
Lipschitz continuity of $F$.
However, the difference quotients 
$\frac{1}{s}\bigl(F(\gamma+s\sigma)-F(\gamma)\bigr)$
of $F$ now have a limit in $\mathcal{M}(S^1\times\R_{\ge })$ with
respect to the weak$^*$ topology:
One has
\[
\frac{1}{s}\bigl(F(\gamma+s\sigma)-F(\gamma)\bigr)
\rightharpoonup^* \gamma^\#(\sigma\mathcal{L}^1),
\]
where $\gamma^\#(s\mathcal{L}^1)$ denotes the push forward of 
the one-dimensional Lebesgue measure $\mathcal{L}^1$ weighted by
the functions $\sigma$ via the mapping $\gamma$.
That is,
\[
\gamma^\#(\sigma\mathcal{L}^1)(h)
= \int_{S^1} \sigma(t)\,h\bigl(t,\gamma(t)\bigr)\,dt
\]
for every $h \in C_b(S^1\times \R_{\ge 0})$.

The considerations above show that the setting of the problem can be modified
in such a way that the operator $F$ becomes differentiable.
Thus it might still be possible to derive convergence rates
in a more classical way via a linearization of $F$.
There are, however, several difficulties:
First, the extension to $\mathcal{M}(S^1\times\R_{\ge 0})$ leaves
the Hilbert space setting in favor of a more complicate Banach space
setting (with non-reflexive spaces).
In such a setting, non-standard convergence rates have been
derived via linearization ideas in~\cite{HeiHof09}
by the technique of approximate source conditions.
Second, even the Banach space setting might not be the correct one,
as the operator $F$ is only differentiable with respect to the
weak$^*$ topology and not the norm topology.
Thus it might even be necessary to derive results in general
locally convex spaces.
Finally note that after changing the topological structure,
the similarity term is not the squared norm on the target space any more.
Thus it might even be necessary to turn to results for convergence
rates with more general similarity terms
(see for instance~\cite{Fle10,Gra10b,Poe08}).

\subsection{Differentiability of the Regularization Term}

Interestingly, the regularization functional itself has potentially much better
smoothness properties than $F$.
In order to see this, define for $u \in L^2(S^1\times\R_{\ge 0})$ 
the functional $\mathcal{S}\colon H^1_+(S^1) \to \R_{\ge 0}$,
\[
\mathcal{S}_u(\gamma) := \norm{F(\gamma)-u}_{L^2}^2.
\]
Then it is easy to see that $\mathcal{S}_u$ is Fr\'echet differentiable
whenever $u \in L^2(S^1\times \R_{\ge 0}) \cap C(S^1\times\R_{\ge 0})$
with Fr\'echet derivative
\[
\mathcal{S}_u'(\gamma)
= \Bigl[\sigma \mapsto \int_{S^1} \sigma(t)\bigl(1-2u(t,\gamma(t))\bigr)\,dt\Bigr].
\]
In addition, for $u = F(\gamma)$ the mapping $\mathcal{S}$, while not
differentiable, possesses one-sided directional derivatives
of the form
\[
\mathcal{S}_{F(\gamma)}'(\gamma;\sigma)
= 2\int_{S^1}\abs{\sigma(t)}\,dt
= 2\norm{\sigma}_{L^1}.
\]

The specific form of this one-sided derivative
might possibly yield an explanation of the linear convergence rate
that has been shown in Corollary~\ref{co:rate} for the case where
$\gamma^\dagger \in W^{2,\infty}(S^1)$, as such
a rate is not possible for smooth operators using the squared
Hilbert space norm as a regularization term.
For quadratic regularization of Fr\'echet differentiable
operators, it has been shown that the best possible convergence rate in
non-trivial cases is of order $O(\delta^{2/3})$ (see~\cite{Neu89}).
Linear convergence rates, however, have been derived recently
for Tikhonov regularization with non-smooth regularization terms.
The first result in this direction was~\cite{GraHalSch08},
where linear rates have been derived for regularization 
on the sequence space $\ell^2$ with the $\ell^1$-norm as regularization
term under the assumptions of sparsity of the true solution
and a certain source condition.
This result has also been extended in~\cite{Gra11a}
to more general positively homogeneous regularization terms.
The basis of all these results is the non-smoothness
of the regularization term; the rates derived in this paper
indicate that a sufficient non-smoothness of the operator $F$
to be inverted can have similar effects.

\section{Conclusion}

We have shown in this paper that the approach of source inequalities
for the derivation of convergence rates for Tikhonov regularization
can be applied to non-linear problems where approaches based
on linearization are bound to fail.
Our paradigm was a functional related to the Chan--Vese active contour
model for image segmentation, which was shown to be continuous but nowhere
differentiable in its domain.
Still, the approach by source inequalities allowed us to derive
convergence rates under reasonable (and easily interpretable)
smoothness assumptions on the true solution.
Surprisingly, the convergence rates do not obey the classical
bound of $O(\delta^{2/3})$, which is known to be best possible rate
for quadratic regularization of differentiable functionals.
Instead, for a sufficiently smooth true solution, we were
able to obtain a rate of order $O(\delta)$.
One possible explanation for this exceedingly good behaviour
is a connection to sparse (or $\ell^1$) regularization:
For noise free data the similarity term in our problem
has precisely the same behaviour as the $\ell^1$-regularization term.

\end{document}